\documentclass[bjps,preprint]{imsart}

\usepackage{amscd}
\usepackage{amsfonts}
\usepackage{amsmath}
\usepackage{amssymb}
\usepackage{amstext}
\usepackage{amsthm}
 \usepackage{color}
\usepackage{graphicx}
\usepackage{hyperref}%
\hypersetup{colorlinks=false, linkbordercolor={1 1 1}, citebordercolor={1 1 1}, urlbordercolor={1 1 1}} 
\usepackage[latin1]{inputenc}
\usepackage{mathrsfs}
\usepackage{mathtools}
\usepackage{pifont}
\usepackage{psfrag}
\usepackage[active]{srcltx}
\usepackage[normalem]{ulem}
\usepackage{vruler}

\newcommand{\N}{\mathbb{N}}
\newcommand{\Z}{\mathbb{Z}}
\newcommand{\R}{\mathbb{R}}

\newcommand{\x}{{\textnormal{\emph{\textbf x}}}}
\newcommand{\y}{{\textnormal{\emph{\textbf y}}}}
\newcommand{\z}{{\textnormal{\emph{\textbf z}}}}

\newcommand{\T}{{x}}

\newcommand{\cX}{{\mathcal X}}

\newcommand{\cQ}{{\mathcal Q}}

\newcommand{\Q}{{Q}}

\renewcommand{\T}{T}

\newcommand{\hU}{{\hat U}}
\newcommand{\hT}{{\hat T}}
\newcommand{\hY}{{\hat Y}}
\renewcommand{\S}{{S}}
\newcommand{\unu}{{\underline\nu}}

\theoremstyle{plain}
\newtheorem{theorem}{Theorem}
\newtheorem{lemma}[theorem]{Lemma}
\newtheorem{proposition}[theorem]{Proposition}
\newtheorem{corollary}[theorem]{Corollary}

\theoremstyle{definition}

\theoremstyle{remark}

\newtheorem*{claim*}{Claim}

\renewcommand{\le}{\leqslant}
\renewcommand{\ge}{\geqslant}
\renewcommand{\prec}{\preccurlyeq}

\RequirePackage[T1]{fontenc}
\usepackage{amsthm,amsmath}
\RequirePackage{hyperref}

\doi{10.1214/14-BJPS269}
\arxiv{arXiv:1410.1976}

\startlocaldefs
\numberwithin{equation}{section}
\theoremstyle{plain}

\endlocaldefs

\begin{document}

\begin{frontmatter}
\title{Yaglom limit via Holley inequality}
\runtitle{Yaglom limit via Holley inequality}
\begin{aug}
\author{\fnms{Pablo A.} \snm{Ferrari}\ead[label=e1,url]{http://mate.dm.uba.ar/\string~pferrari/}\ead[label=e1]{pferrari@dm.uba.ar}}
\and
\author{\fnms{Leonardo T.} \snm{Rolla}\ead[label=e2,url]{http://mate.dm.uba.ar/\string~leorolla/}\ead[label=e2]{leorolla@dm.uba.ar}}
\runauthor{P.A. Ferrari and L.T. Rolla}
\affiliation{Universidad de Buenos Aires}
\address{Departamento de Matemática\\
Facultad de Ciencias Exactas y Naturales\\
Universidad de Buenos Aires\\
Pabellón 1, Ciudad Universitaria\\
1428 Ciudad de Buenos Aires, Argentina\\
\printead{e1,e2}}
\end{aug}
\begin{abstract}
Let $\S$ be a countable set provided with a partial order and a minimal element.
Consider a Markov chain on $S\cup\{0\}$ absorbed at $0$ with a quasi-stationary distribution.
We use Holley inequality to obtain sufficient conditions under which the following hold.
The trajectory of the chain starting from the minimal state is stochastically dominated by the trajectory of the chain starting from any probability on~$\S$, when both are conditioned to nonabsorption until a certain time.
Moreover, the Yaglom limit corresponding to this deterministic initial condition is the unique minimal quasi-stationary distribution in the sense of stochastic order.
As an application, we provide new proofs to classical results in the field.
\end{abstract}
\end{frontmatter}

This preprint has the same numbering of sections, equations and theorems as the the
published article
``\emph{Braz. J. Probab. Stat. 29 (2015), 413--426.}''

\section{Introduction}
Let $S$ be a countable set, $0$ be an element outside $S$ and consider a Markov chain $X_n\in S\cup\{0\}$ with transition matrix $\Q(x,y)$, $x,
y\in S\cup\{0\}$.
We assume that the matrix $Q$ yields a chain which is \emph{absorbed at $0$},
meaning that $\Q(0,0)=1$.
We assume also $Q(x,0)<1$ for all $x\in \S$.

For a probability measure $\nu$ on $\S$, we define $\nu T_n$ as the conditional distribution at time~$n$ of the chain started with law $\nu$ given that it is not absorbed until time~$n$.
More precisely,
\begin{equation}
\label{b10}
\nu T_n(y) :=\frac{\nu \Q^n(y)}{1-\nu Q^n(0)},\qquad y \in \S.
\end{equation}
A probability measure $\nu$ is a \emph{quasi-stationary distribution} (or simply \emph{qsd}) if $\nu=\nu T_1$ (and thus $\nu T_n=\nu$
for all $n\ge 1$).
Rewriting~(\ref{b10}), a measure $\nu$ is a qsd if and only if
\begin{equation}
\label{b11}
\nu(y) =\sum_{x \in S} \nu(x) \big[\Q(x,y)+\Q(x,0)\nu(y)\big],\qquad y\in\S.
\end{equation}

The \emph{Yaglom limit} of $\nu$ is $\lim_n\nu T_n$, if the limit exists and is a probability measure.
In this case the limit is called a \emph{quasi-limiting distribution}.

Assume that $\S$ has a partial order denoted $\le$ and a minimal state called $1$.
Let $\nu\prec\nu'$ denote the stochastic order of measures on $\S$ induced by $\le$.

We say that a probability measure $\mu$ on $\S^n$ is \emph{irreducible} if the set
$\{(x_1,\dots,x_n)\in \S^n: \mu(x_1,\dots,x_n)>0\}$ is connected in the sense
that any element of $S^n$ with positive $\mu$-probability can be reached from
any other via successive coordinate changes without passing through elements
with zero $\mu$-probability. We say that a Markov chain on $\S\cup\{0\}$ with
transition matrix $\Q$ and initial distribution $\nu$ on $\S$
has \emph{irreducible trajectories in~$\S$} if for each $n\ge
1$ the measure $\mu$ on $\S^{n+1}$ defined by
\(
\mu(x_0,\dots,x_n) :=
 \nu(x_0)Q(x_0,x_1)\dots Q(x_{n-1},x_n)
\)
is irreducible.

Let $\delta_x$ be the probability distribution on $\S$ concentrated on the state $x\in\S$.
\begin{theorem}
\label{t1}
Let $\S$ be a partially ordered countable space with a minimal element called 1. Let $Q$ be the transition matrix of a Markov chain on $\S\cup\{0\}$ absorbed at~$0$.
Assume that the chain with initial distribution $\delta_1$ has irreducible trajectories in $\S$. If, for all $x,x',z,z'\in\S$ with $x\le x'$, $z\le z'$, whenever the denominators are positive,
\begin{equation}
\label{e2}
\displaystyle{\frac{\Q(x,\cdot)\Q(\cdot,z)}{\Q^2(x,z)}} \prec
\displaystyle{\frac{\Q(x',\cdot)\Q(\cdot,z')}{\Q^2(x',z')}},
\end{equation}
\begin{equation}
\label{e3}
\frac{\Q(x,\cdot)}{1-\Q(x,0)} \prec \frac{\Q(x',\cdot)}{1-\Q(x',0)},
\end{equation}
as probability measures on $\S$, then the following hold:\\
i. The sequence $(\delta_1 T_{n})_{n\ge 1}$ is monotone:
$\delta_1 T_n\prec\delta_1 T_{n+1}$, for all $n\ge 0$.\\
ii. For any probability $\nu$ on $\S$, $\delta_1 T_n\prec\nu T_n$.\\
iii. In particular, if $\nu$ is a qsd, then $\delta_1 T_n\prec\nu$, for all $n\ge 0$.\\
iv. If there is a qsd for $\Q$, then the Yaglom limit of $\delta_1$ converges.
The limit distribution $\unu := \lim_n\delta_1T_n$ is a qsd and satisfies $\unu \prec \nu$ for any other qsd $\nu$.
\end{theorem}

The proof of Theorem~\ref{t1} is an application of Holley inequality in the space of finite-length trajectories of the chain.
Roughly speaking, Holley inequality says that the local dominations~(\ref{e2}) and~(\ref{e3}) imply that the conditional law of a length-$n$ trajectory of the chain starting with $\delta_1$ given nonabsorption by time~$n$ is stochastically dominated by the conditional law of a length-$n$ trajectory of the chain starting from any other measure~$\nu$.

In Section~\ref{sec:yaglom}, we state and prove Holley inequality, and use it to prove Theorem~\ref{t1}.
We then give a sufficient condition for~$\unu$ to be the minimal qsd in the sense of absorption time rather than stochastic domination.

Convergence of the Yaglom limit has been studied for birth-and-death chains and one-dimensional random walks in both continuous and discrete time.
Theorem~\ref{t1} gives an alternative proof to many of these classical results.

References to previous works and details of our approach to the aperiodic cases are discussed in detail in Section~\ref{s4}.
Periodic chains are discussed in Section~\ref{sec:periodic}, after condition~(\ref{e2}) is relaxed so as to include this case.

There is a large literature on qsd's compiled and periodically updated by Pollett~\cite{pollett}.
We quote the recent book of Collet, Martínez and San Martín~\cite{MR2986807} and the work of Kesten~\cite{MR1341881} on Yaglom limits of discrete-time Markov chains.

\section{Yaglom limit via Holley inequality}
\label{sec:yaglom}

\subsection{Trajectory distribution}

For integers $n<m$ let $\cX_n^m :=\{(x_n,\dots,x_m)\,:\,x_k\in\S$ for $n\le k\le m\}$ be the set of possible trajectories of the chain with transition matrix $\Q$ in the time interval $[n,m]$ which are not absorbed by
0 in that time interval.

Let $\nu$ be a probability measure on $\S$ and define the measure
$\mu_n^m(\nu,\Q)$ on~$\cX_n^m$~by
\begin{equation}
\label{a10}
\mu_n^m(\nu,\Q)(\x_n^m):=
\frac{\nu(x_n)\Q(x_n,x_{n+1})\dots \Q(x_{m-1},x_m)}{1-\nu \Q^{n}(0)}
\end{equation}
where $\x_n^m=(x_n,\dots,x_m)$. The measure $\mu_n^m(\nu,\Q)$ is the conditional distribution of the chain $X_n^m=(X_n,\dots,X_m)$ with initial distribution $\nu$ at time~$n$ and transition probabilities $\Q$, given that the chain is not absorbed during the time interval $[n,m]$.

Due to the conditioning, the first marginal of~$\mu_n^m(\nu,\Q)$ is not $\nu$ in general, but its last marginal is $\nu T_{m-n}$.
Indeed, by~(\ref{b10}) and~(\ref{a10}),
\begin{equation}
\label{b12}
\nu T_{m-n}(y) = \sum_{\quad\quad\mathclap{(x_n,\dots,x_{m-1})}\quad}\mu_n^m(\nu,\Q)(x_n,\dots,x_{m-1},y).
\end{equation}

\subsection{Holley inequality}

Let $\Omega$ be a set endowed with a partial order $\le$.
Let $\mu, \mu'$ be probability measures on $\Omega$.
The stochastic domination $\mu\prec\mu'$ is equivalent to the existence of a measure $\tilde \mu$ on $\Omega\times\Omega$ with marginals $\mu$ and $\mu'$ such that $\tilde\mu((\omega,\omega'):\omega\le \omega')=1$, see for instance~\cite{MR1711599}.
In this case we say that $\tilde \mu$ is a \emph{monotone coupling} of $\mu$ and $\mu'$.

Let us endow~$\cX_n^m$ with the partial order given by the coordinate-wise order of trajectories: $\x_n^m\le \y_n^m$ if $x_k\le y_k$ for all $k\in[n,m]$.

A local domination condition for global domination of
measures is provided by Holley Inequality~\cite{holley}. Here is a version suited to our context.

\begin{proposition}[Holley inequality]
\label{holley}
Let $\S$ be a partially ordered countable space. Let $\nu\prec\nu'$ be
probabilities on $\S$ and let $\Q$, $\Q'$ be transition matrices on
$\S\cup\{0\}$ absorbed at $0$.
Denote the conditional laws of trajectories by $\mu=\mu_n^m(\nu,\Q)$ and $\mu'=\mu_n^m(\nu',\Q')$, respectively.
Assume that $\mu$ is an irreducible probability on the space of trajectories $\cX_n^m$.
If, for all $x,x',z,z'\in\S$ with $x\le x'$, $z\le z'$, whenever the denominators are positive,
\begin{align}
\tag{a}
\frac{\nu(\cdot)\Q(\cdot,z)}{\nu \Q(z)} & \prec \frac{\nu'(\cdot)\Q'(\cdot,z')}{\nu' \Q'(z')}
\\[2mm]
\tag{b}
\frac{\Q(x,\cdot)\Q(\cdot,z)}{\Q^2(x,z)}& \prec \frac{\Q'(x',\cdot)\Q'(\cdot,z')}{\Q'^2(x',z')}
\\[2mm]
\tag{c}
\displaystyle \frac{\Q(x,\cdot)}{1-Q(x,0)} & \prec \frac{\Q'(x',\cdot)}{1-Q'(x',0)}
\end{align}
as measures on $\S$, then $\mu\prec\mu'$.
\end{proposition}
Holley inequality was proved in~\cite{georgii-haggstrom-maes} for finite state space~$S$.
We use the Markovian structure of $\mu$ and condition~(c) to get around this assumption.

\begin{proof}
Let $(\eta_t:t\ge0)$ be the Gibbs sampler for $\mu$, a Markov jump process on $\cX_n^m$ with the following evolution:
the value at each site $k\in[n,m]$ is updated at rate $1$ to a new value using the conditional distribution of $\mu$ given the configuration at the sites $[n,m]\setminus\{k\}$.
Different sites are never updated simultaneously since they use independent Poisson clocks.
This amounts to use the measures in the left hand side of (a), (b) and (c) to update sites~$n$, $[n+1,m-1]$ and~$m$, respectively.
The measure $\mu$ is reversible for
$\eta_t$.
Analogously, let $(\eta'_t:t\ge0)$ be the Gibbs sampler for $\mu'$ for which the updating is done with the measures in the right hand side of~(a),~(b) and~(c), respectively.
The measure~$\mu'$ is reversible for $\eta'_t$.

We will use the stochastic inequalities (a,b,c) to construct a monotone coupling $((\eta_t,\eta'_t):t\ge0)$ of both Gibbs sampler processes.
In this coupling, at rate 1 the value at each site in $[n,m]$ is simultaneously updated for both marginal trajectories using a monotone coupling of the measures in (a), (b) and (c) to update sites $n$, $[n+1,m-1]$ and $m$, respectively.
If the trajectories are ordered at time 0, then they will remain ordered at future times, that is, if $\eta_0\le\eta'_0$ then $\eta_t\le\eta'_t$ for all $t\ge 0$.
We thus need to find $\eta_0\le\eta'_0$.

We claim that, given any trajectory $\z_n^m$ with positive $\mu'$-probability, there exists a trajectory $\x_n^m \le \z_n^m$ with positive $\mu$-probability.
We prove this by constructing~$\x_n^m$ as follows.
Since $\nu'(z_n)>0$ and $\nu \prec \nu'$, it is possible to choose~$x_n \le z_n$ such that $\nu(x_n)>0$.
Suppose that $x_n \le z_n,\dots,x_k \le z_k$ have been chosen.
Since $Q'(z_k,z_{k+1})>0$, from condition~(c) it is possible to choose $x_{k+1}\le z_{k+1}$ such that $Q(x_k,x_{k+1})>0$.
This proves the claim.

The coupled process starts from $(\eta_0,\eta'_0)$, where $\eta'_0$ is distributed with the reversible measure $\mu'$ and, given $\eta'_0$, a trajectory $\eta_0\le\eta'_0$ with positive $\mu$-probability is chosen according to the previous claim.
Since $\eta_0\le\eta'_0$, by the coupling we have $\eta_t\le \eta'_t$ for all $t\ge 0$, and thus the law of $\eta_t$ is stochastically dominated by that of $\eta'_t$.
Irreducibility implies that~$\eta_t$ converges in distribution to its unique invariant measure~$\mu$.
On the other hand, the distribution of $\eta'_t$ is $\mu'$ for all $t$.
Letting $t\to\infty$, we get $\mu \prec \mu'$.
\end{proof}

The proof of Holley inequality works also for the nonhomogeneous case. Consider a family o transition matrices $\mathcal Q =(Q_k,k\in \Z)$ and let $(X_k)$ be a (nonhomogeneous) Markov chain satisfying $P(X_{k+1}=y|X_k=x)= Q_k(x,y)$, that is, the transition matrix $\Q_k$ is used at time $k$.
Then we have the following corollary of the proof of Proposition~\ref{holley}.

\pagebreak[3]
\begin{corollary}
\label{nonhomogeneous}
Proposition~\ref{holley} holds for nonhomogeneous families of transition matrices $\cQ$ and $\cQ'$ such that $Q_n$ and $Q_n'$ satisfy~\textnormal{(a)}, $Q_{m-1}$ and $Q_{m-1}'$ satisfy~\textnormal{(c)}, and that, for $k=n+1,\dots,m-1$,
\[
\tag{b'}
\frac{\Q_{k-1}(x,\cdot)\Q_{k}(\cdot,z)}{\Q_{k-1}Q_k(x,z)}
\prec
\frac{\Q_{k-1}'(x,\cdot)\Q_{k}'(\cdot,z)}{\Q_{k-1}'Q_k'(x,z)}
.
\]
\end{corollary}

\subsection{Monotonicity and Yaglom limit}

\begin{proof}
[Proof of Theorem~\ref{t1}]
We first prove (ii) using Holley inequality with $Q'=Q$ and $\nu=\delta_1$. The probability measure on the left-hand side of condition~(a) is $\delta_1$, and, since $1$ is minimal in $\S$, this condition is satisfied for any $\nu'$ on $\S$.
Conditions~(b) and~(c) are being assumed in~(\ref{e2}) and~(\ref{e3}).
Irreducibility of $\mu_{0}^n(\delta_1,\Q)$ has also been explicitly assumed. By Holley inequality, $\mu_{0}^n(\delta_1,\Q)\prec\mu_{0}^n(\nu',\Q)$, which by~(\ref{b12}) implies $\delta_1 T_n\prec\nu' T_n$, concluding the proof of~(ii).

If $\nu'$ is a qsd, then $\nu'=\nu'T_n$.
Together with (ii), this implies (iii).

To prove~(i), we introduce a nonhomogeneous chain $\mathcal Q$ forced to make
the first jump into state 1 while the rest of the jumps are governed by $\Q$.
Let $\mathcal Q=(Q_k,\,-1\le k\le n)$ be given by $\Q_{-1}(x,1)=1$
for all $x\in S$ and $\Q_k=Q$ for $k=0,\dots,n$.
By definition of $\Q_{-1}$, the projection of $\mu_{-1}^n(\delta_1,\mathcal Q)$ onto $\cX_{0}^n$ is $\mu_{0}^n(\delta_1,\Q)$.
Hence, the time-$n$ marginal of $\mu_{-1}^n(\delta_1,\mathcal Q)$ is $\delta_1T_{n}$, the same as the time-$n$ marginal of $\mu_{-n}^0(\delta_1,\Q)$.
Let $\mathcal Q'=(Q'_k,\,-1\le k\le n)$ be given by $Q'_{k}=Q$ for $k=-1,\dots,n$ (that is, the homogeneous chain).
The time-$n$ marginal of $\mu_{-1}^n(\delta_1,\mathcal Q')$ is $\delta_1T_{n+1}$, the same as the time-$n$ marginal of $\mu_{-1}^n(\delta_1,\Q)$.
Again, condition~(c) has been assumed in~(\ref{e3}).
Writing $\nu=\nu'=\delta_1$, condition~(a) holds trivially.
Condition (b') is trivial for $k=0$, and for $k = 1,2,\dots,n-1$ it is assumed in~(\ref{e2}).
Also, irreducibility of $\mu_{0}^n(\delta_1,\Q)$, and thus of $\mu_{-1}^n(\delta_1,\mathcal Q)$, has been explicitly assumed.
Using Corollary~\ref{nonhomogeneous} we get $\mu_{-1}^n(\delta_1,\mathcal Q) \prec \mu_{-1}^n(\delta_1,\mathcal Q')$, and thus $\delta_1 T_n\prec \delta_1 T_{n+1}$, proving (i).

To show (iv), let $\nu'$ be a qsd.
By (i), $\delta_1 T_n$ is an increasing sequence of measures and by (iii) all elements of the
sequence are dominated by $\nu'$.
Hence there is a limit $\unu:=\lim_n\delta_1 T_n\prec\nu'$.
To check that $\unu$ is a qsd, use that~$T_n$ is a semigroup and that $T_1$ is continuous to get
\[
\unu= \lim_n \delta_1T_{n+1} = \lim_n \delta_1T_{n}T_1 =\unu T_1.\qedhere
\]
\end{proof}

\subsection{Yaglom limit and minimal qsd}

For a measure $\nu$ on $\S$, denote
\begin{align}
 \label{p35}
 a(\nu):=1-\nu Q(0),
\end{align}
the mass staying at $S$ after one step for the chain starting with $\nu$. If $\nu$ is a qsd, then $\nu$ is a left eigenvector for $\Q_{|_\S}$ with eigenvalue $a(\nu)$:
\begin{align}
\nonumber
\nu\Q_{|_\S} = a(\nu)\nu.
\end{align}
Let $a_{*}:= \inf\{a(\nu):\nu$ is a qsd$\}$.
If there exists a qsd~$\nu$ with $a(\nu)=a_*$, then it is called \emph{minimal} and denoted $\nu_{\min}$.

The following lemma gives sufficient conditions in terms of $\Q$ so that the measure $\unu$ given by Theorem~\ref{t1} coincides with $\nu_{\min}$.

\begin{lemma}
\label{qmon1}
If $Q$ is such that
\(
\Q(x,0) \ge \Q(x',0),\hbox{ for all }x\le x' \in\S,
\)
and $\unu$ is a qsd such that $\unu \prec \nu$ for any other qsd $\nu$, then $\unu=\nu_{\min}$.
\end{lemma}
\begin{proof}
The function $f:\S\to \R^+$ given by $f(y)=Q(y,0)$ is nonincreasing, whence $\unu f \ge \nu f$ for any qsd $\nu$.
Thus, $a(\unu) \le a(\nu)$, and taking the infimum over~$\nu$ we get $a(\unu)=a_*$, which proves the lemma.
\end{proof}

\section{The birth-and-death chain}
\label{s4}

In this section we consider $\S=\N$ with the usual order and birth-and-death processes. The transition matrix $\Q$ is defined by:
\begin{align}
&p_x,\;r_x,\; q_x >0,\; q_x+r_x+p_x=1, \hbox{ for all } x\ge 1; \nonumber\\[2mm]
&\Q(x,x-1) = q_x,\; \Q(x,x) = r_x,\; \Q(x,x+1) = p_x, \label{bd}
\\[2mm]
&\Q(x,y)=0 \hbox{ if }|x-y|>1\hbox{ and }\Q(0,0)=1.\nonumber
\end{align}
In this case there exist a qsd if the absorption time of the chain
starting from a fixed state has an exponential moment; see for
instance van Doorn and Schrijner~\cite[Corollary~4.1]{MR1359177},
Ferrari, Martínez and Picco~\cite[Theorem~6.1]{MR1188953} and
Ferrari, Kesten, Martínez and Picco~\cite{MR1334159}.
Under these conditions, Cavender~\cite{cavender} shows that there is a critical value $\gamma>0$ such that there is a one-parameter family of qsd's $\{\nu:\nu(1)\in(0,\gamma]\}$ indexed by $\nu(1)$.
Cavender fixes $\nu(1)\le \gamma$ and computes explicitly the other values using the equation~(\ref{b11}) and the nearest-neighbor structure (this procedure does not yield a probability if $\nu(1)>\gamma$).
Cavender also shows that any pair of qsd's $\nu$, $\nu'$  satisfy a monotone likelihood ratio:
$\nu'(1)>\nu(1)$ implies $\frac{\nu(1)}{\nu'(1)}< \frac{\nu(2)}{\nu'(2)}\le \frac{\nu(3)}{\nu'(3)}\le \dots$, which in
turn implies the domination $\nu'\prec\nu$.

Van Doorn and Schrijner~\cite{vandoorn-schrijner-ratio} use the Karlin and McGregor polynomial representation of the chain to give a sufficient condition for the Yaglom limit to converge to an explicit limit.
Ferrari, Martínez and
Picco~\cite{ferrari-martinez-picco-dynamical} describe the domain of attraction
of qsd's and show in particular that the Yaglom limit of $\delta_x$ converges to
the minimal qsd, for any initial state~$x$.
Daley~\cite{MR0246374} and Iglehart~\cite{MR0368168} showed the Yaglom limit for random walks with negative drift and and finite variance, respectively for discrete and continuous space.

In the sequel we develop Theorem~\ref{t1}'s conditions and make them explicit for the case of birth-and-death chains.
Corollary~\ref{corrwd3} is about space-homogeneous discrete-time random walks with delay.

Item~(iv) of Corollary~\ref{corseneta} gives the Yaglom limit for continuous-time walks. It was originally proven by Seneta~\cite{seneta} using direct computation.
Our proof uses monotonicity of the trajectories instead.

Corollary~\ref{c10}, presented in the next section, gives the Yaglom limit for the discrete-time periodic chain.
It provides an alternative proof to that of Seneta and Vere-Jones~\cite{seneta-verejones}.

\paragraph{The conditions of Theorem~\ref{t1}}

Since the state space $\S=\N$ is totally ordered and the transitions are only to nearest neighbors, we can obtain conditions~(\ref{e2}) and~(\ref{e3}) in explicit terms of $p_k$, $r_k$ and $q_k$.
Take~$\Q$ as defined in~(\ref{bd}).
Define for positive integers $x,z,y$:
\begin{align}
b((x,z),y)&:=\sum_{w\ge y}\frac{\Q(x,w)\Q(w,z)}{\Q^2(x,z)};\label{bbb}\\
c(x,y)&:=\sum_{w\ge y}\frac{\Q(x,w)}{1-\Q(x,0)}.\label{ccc}
\end{align}
Conditions~(\ref{e2}) and~(\ref{e3}) are equivalent to
\begin{align}
b((x,z),y)&\le b((x',z'),y), \quad \hbox{ for } z\le z',\; x\le x'; \label{hb}\\
c(x,y)&\le c(x',y), \quad \quad \quad \hbox{ for } x\le x',\label{hc}
\end{align}
for all $y\ge 1$, whenever the denominators of both sides are positive.

Inequalities~(\ref{hb}) hold trivially when $y=1$ or $\{x,x',z,z'\}\not\subset\{y-1,y\}$. The remaining cases are the following.
For $y\ge 2$ the conditions~(\ref{hb}) are equivalent to the following conditions:
\begin{equation}
\begin{aligned}
\label{bb13}
 b((y-1,y-1),y)\le b((y,y-1),y)\le b((y,y),y), \\
 b((y-1,y-1),y)\le b((y-1,y),y)\le b((y,y),y).
\end{aligned}
\end{equation}
Using the convention $p_0=0$, conditions~(\ref{bb13}) for $y\ge 2$ read
\begin{equation}
\label{bb2}
\mathclap{
\begin{aligned}
\frac{p_{y-1}q_y}{r_{y-1}^2+p_{y-1}q_y+q_{y-1}p_{y-2}} &\le
\frac{r_yq_y}{q_yr_{y-1}+r_yq_y}
\le \frac{r_y^2+p_yq_{y+1}}{r_y^2+p_yq_{y+1}+q_yp_{y-1}},\\
\frac{p_{y-1}q_y}{r_{y-1}^2+p_{y-1}q_y+q_{y-1}p_{y-2}} &\le
\frac{p_{y-1}r_y}{p_{y-1}r_y+r_{y-1}p_{y-1}} \le
\frac{r_y^2+p_yq_{y+1}}{r_y^2+p_yq_{y+1}+q_yp_{y-1}}.
\end{aligned}
\quad }
\end{equation}

Analogously, conditions~(\ref{hc}) on $c(x,y)$ hold trivially when $y=1$ or $(x,x')\neq (y-1,y)$.
Hence,~(\ref{hc}) is equivalent to
\begin{align}
\nonumber
c(y-1,y)\le c(y,y),\quad y \ge 2,
\end{align}
which in the case $y=2$ and $y\ge 3$ read, respectively,
\begin{align}
\label{bb3}
 \displaystyle{\frac{p_1}{p_1+r_1}}\le r_2+p_2,\qquad p_{y-1}\le p_{y}+r_{y}, \hbox{ for } y\ge3.
\end{align}
We summarize these computations as a lemma.
\begin{lemma}
 \label{ll1}
Let $\Q$ be the transition matrix for the birth-and-death chain defined in~\textnormal{(\ref{bd})}.
Then conditions~\textnormal{(\ref{e2})} and~\textnormal{(\ref{e3})} are equivalent to~\textnormal{(\ref{bb2})} and~\textnormal{(\ref{bb3})}.
\end{lemma}

We are ready to state the result in this case.
\begin{corollary}
\label{corbd}
Assume that the birth-and-death chain absorbed at zero defined in~\textnormal{(\ref{bd})} has at least one qsd and satisfies conditions~\textnormal{(\ref{bb2})} and~\textnormal{(\ref{bb3})}.
Then~\textnormal{(i,ii,iii,iv)} of Theorem~\textnormal{\ref{t1}} hold.
The Yaglom limit of $\delta_1$ coincides with~$\nu_{\min}$, the minimal qsd in the sense of absorption time.
Furthermore, for any $x\in\N$, the Yaglom limit of $\delta_x$ also converges to $\nu_{\min}$.
\end{corollary}

\begin{proof}
Since for $x\ge 1$ the probability of transitions from $x$ to $x$ and to nearest neighbors of $x$ are positive, the birth-and-death chain starting with~$\delta_1$ has irreducible trajectories in $\N$.
By Lemma~\ref{ll1}, the conditions of Theorem~\ref{t1} are equivalent to the present conditions, hence (i,ii,iii,iv) of Theorem~\ref{t1} hold and the Yaglom limit of $\delta_1$ converges to $\unu$.
Since $Q(x,0)=0$ for all $x>1$, Lemma~\ref{qmon1} applies and $\unu=\nu_{\min}$.
By~\cite[Theorem~3.1]{ferrari-martinez-picco-dynamical}, if the Yaglom limit of $\delta_1$ exists, then it coincides with the Yaglom limit of $\delta_x$ for any $x$, concluding the proof.
\end{proof}

\subsection{Random walk with delay}
The absorbed delayed random walk is a particular case of birth-and-death chain on $\N\cup\{0\}$ defined in~(\ref{bd}) with constant transition probabilities along~$\N$:
\begin{equation}
\label{eqrwdelay}
\begin{array}[l]{l}p_x\equiv p,\; q_x\equiv q,\; r_x\equiv r,\\
 p,q,r>0,\; p+q+r=1, \; p<q.
\end{array}
\end{equation}
This walk has a drift towards 0 and it is absorbed at 0.
A probability~$\nu$ on~$\N$ is a qsd if and only if it satisfies the equations~(\ref{b11}), which in this case are
\begin{align}
\label{qsdpqr}
\nu(x+1)q+\nu(x-1)p+(q\nu(1)-(p+q))\nu(x)= 0, \quad x\ge 1,
\end{align}
with the convention $\nu(0)=0$.
Cavender~\cite{cavender} proved that the set of qsd's is a family indexed by $\nu(1)$ with $\nu(1)\in (0, (1-\sqrt{\lambda})^2]$, where $\lambda=p/q$.
Since the absorption probability of a qsd $\nu$ is $\nu\Q(0)= q\nu(1)$, the qsd with maximal~$\nu(1)$ is the minimal qsd $\nu_{\min}$, a negative binomial with parameters~$2$ and~$\sqrt{\lambda}$:
\begin{equation}
\label{a18}
\nu_{\rm min}(x) = \big(1-\sqrt\lambda\big)^2 x \big(\sqrt\lambda\big)^{x-1} ,\quad x\ge1.
\end{equation}
The remaining qsd are given in function of $\nu(1)\in\big(0,\big(1-\sqrt\lambda\big)^2\big)$ by
\begin{equation}
\label{qsd}
\nu(x) =
\frac{\nu(1)}{c}\Bigl[
\Bigl(\frac{\lambda + 1-\nu(1)+c}{2}\Bigr)^x
-\Bigl(\frac{\lambda+1-\nu(1)-c}{2}\Bigr)^x\Bigr],
\end{equation}
where $c=[(\nu(1)-\lambda-1)^2 -4\lambda]^{1/2}$, \cite[p.~585]{cavender}.

\begin{corollary}
\label{corrwd3}
Consider the random walk with delay absorbed at zero defined in~\textnormal{(\ref{bd})} with constant rates~\textnormal{(\ref{eqrwdelay})}.
If $pq\le r^2$, then the conclusions \textnormal{(i,ii,iii,iv)} of Theorem~\ref{t1} hold with $\unu=\nu_{\min}$ given by~\textnormal{(\ref{a18})}.
Furthermore, for any $x\in\N$ the Yaglom limit of $\delta_x$ converges to $\unu$.
\end{corollary}

\begin{proof}
In the present context the worst case of~(\ref{bb2}) is when $y=2$, which reduces to:
\begin{align}
\frac{pq}{pq+r^2} \le \frac12 \le \frac{r^2+pq}{r^2+2pq}.
\label{f14}
\end{align}
On the other hand~(\ref{bb3}) reads
\begin{align}
\frac{p}{p+r}\le r+p&, \qquad
p\le p+r. \label{f16}
\end{align} Condition $pq\le r^2$ implies both~(\ref{f14}) and~(\ref{f16}).
The result thus follows from Corollary~\ref{corbd}.
\end{proof}

If $r<\sqrt{pq}$, then trajectory domination is not true.
Although the Yaglom limit of~$\delta_1$ is known to hold in this case~\cite{MR0246374}, it does not seem to follow from the arguments presented here, except for the periodic case $r=0$ discussed in Section~\ref{sec:periodic}.

\subsection{The continuous-time random walk}

Take positive $p<q$ with $p+q=1$ and
consider a family of random walks with delay $(X^r_n)$, indexed by $r\in[0,1)$, with
transition probabilities
\begin{equation}
\nonumber
 \begin{array}[l]{l}
 \Q_r(x,x-1)=q(1-r),\; \Q_r(x,x)= r,\; \Q_r(x,x+1)=p(1-r),\;\\
 \Q_r(x,y)=0,\text{ otherwise, } x\ge 1;\quad \Q_r(0,0)=1.
 \end{array}
\end{equation}
Define the rescaled process
\[
Y^r_t := X^r_{[t/(1-r)]}.
\]
As $r \to 1$, the process $(Y^r_t)$ converges in finite time-intervals to the
process $(\hat Y_t)$, a continuous-time random walk with rates $p,q$ to jump one unit
forward or backwards, respectively, and absorbed at $0$.
Call $\hU_t$ the
corresponding semigroup:
\[
\hU_t(x,y):=P(\hat Y_t=y|\hat Y_0=x).
\]
Define $\nu T_t^r$ as the probability given by
\[
\nu T_t^r(y):= \frac{\nu \Q_r^{[t/(1-r)]}(y)}{1-\nu Q_r^{[t/(1-r)]}(0) },
\]
that is, $\nu T_t^r$ is the distribution at time $t$ of the walk $Y^r_t$ starting with $\nu$, conditioned to nonabsorption.
This distribution converges as $r\to 1$ to the distribution at time $t$ of the continuous-time walk $\hat Y_t$ under the same condition:
\begin{equation}
\label{cl}
\lim_{r\to 1} \nu T^r_t(y) = \nu\hat T_t(y) := \frac{\nu \hU_t(y)}{1-\nu \hU_t(0) }.
\end{equation}
The resulting operator $\hat T_t$ is a semigroup. For any $r\in[0,1)$, the qsd's for~$Y^r_t$ satisfy equations~(\ref{qsdpqr}) because the factors $(1-r)$ cancel out.
Moreover, the qsd's for the continuous-time walk~$\hY_t$ also satisfy the same
equations.
Indeed, $\nu=\nu\hat T_t$ if and only if
\(
\nu (\hU_t-I) + \nu \hU_t(0)\cdot\nu=0
;
\)
dividing by~$t$ and letting $t\to 0$ yields~(\ref{qsdpqr}).
As a consequence, the minimal qsd for both $Y^r_t$ and $\hY_t$ is given by~(\ref{a18}) while the remaining qsd are given by~(\ref{qsd}). In the continuous-time case, $p$ and $q$ may be any positive real numbers satisfying $p<q$; the definitions~(\ref{a18}) and~(\ref{qsd}) depend on $p$ and~$q$ only through the ratio $\lambda=p/q$.

\pagebreak[3]
\begin{corollary}
\label{corseneta}
The continuous-time random walk with rates $p,q$ absorbed at zero
satisfies:\\
i. The sequence $(\delta_1 \hat T_{t}, t \ge 0)$ is monotone: $\delta_1
\hat T_s\prec\delta_1 \hat T_{t}$ for $0<s\le t<\infty$.
\\
ii. If $\nu$ is a probability measure on $\N$, then
$\delta_1 \hat T_t\prec\nu\hat T_t$ for all $t\ge 0$.
\\
iii. In particular, if $\nu$ is a qsd, then $\delta_1 \hat T_t\prec\nu$.
\\
iv. The Yaglom limit of $\delta_1$ converges to $\nu_{\rm min}$ given
by~\textnormal{(\ref{a18})}.
\end{corollary}
\begin{proof}
Take $r$ sufficiently close to one so that $pq(1-r)^2\le r^2$, to be under the conditions of Corollary~\ref{corrwd3}.

To show (i) we use Corollary~\ref{corrwd3}(i) to get
\(
\delta_1 T^r_t\prec\delta_1 T^r_{t+s},
\)
for all $t,s\ge0$,
and then use~(\ref{cl}) to conclude.
To prove (ii,iii), we use Corollary~\ref{corrwd3}(ii,iii) to get $\delta_1 T^r_t\prec\nu T^r_t$ (which equals $\nu$ if it is a qsd), and again use~(\ref{cl}) to conclude.

Let us show (iv).
As discussed above, $\nu_{\min}$ given by~(\ref{a18}) is a qsd, the other qsd's are given by~(\ref{qsd}), and in particular $\nu_{\min}$ is minimal also in the sense of stochastic ordering.
By (i,iii), there is $\unu = \lim_{t} \delta_1\hT_t$.
As in the proof of Theorem~\ref{t1}, using the semigroup property of $\hT_t$, the limit $\unu$ is a qsd.
It follows from (iii) that $\unu \prec \nu_{\min}$, and therefore $\unu=\nu_{\min}$.
\end{proof}

\section{The periodic case}
\label{sec:periodic}

Assume that the matrix $Q$ is irreducible in $\S$ and that $\Q$ restricted to $\S$ has period $d\ge 2$. Let $\S_1,\dots,\S_{d}\subset \S$ be the cyclic subclasses, that is, the equivalence classes induced by the equivalence relation $\sim$ defined by $x\sim y$ if and only if $Q^{d \ell}(x,y)>0$ for some $\ell\ge 1$.
Assume that the classes are labeled so that $x\in\S_j$, $Q(x,y)>0$ implies $y\in \S_{j+1}\cup\{0\}$ (with the convention $\S_{d+1}=\S_1$).

\pagebreak[3]
\begin{theorem}
\label{t7}
Let $\S$ be a partially ordered countable set with a minimal element called $1$ and let $X_n$ be a Markov chain on $\S\cup\{0\}$, absorbed at $0$, irreducible in $\S$ and with period $d$ when restricted to $\S$.
Let $\S_1,\dots,\S_{d}\subset \S$ denote the cyclic subclasses of the chain restricted to $\S$, choosing $S_1\ni 1$.
Assume that the chain with initial state $1$ has irreducible trajectories.

If, for all $x,x',z,z'\in\S$ with $x\le x'$, $z\le z'$, $x$ and $x'$ in the same class, the stochastic inequalities~\textnormal{(\ref{e2})} and~\textnormal{(\ref{e3})} are satisfied whenever the denominators are positive, then the following hold.
\\
i. Monotonicity: $\delta_1\T_{n}\prec \delta_1\T_{n+d}$ for any $n\ge 0$.
\\
ii. For any probability $\nu$ on $\S_1$, one has $\delta_1\T_{n}\prec \nu \T_{n}$.
\\
iii. If $\nu$ is a qsd, then $\delta_1\T_{d k +j-1}\prec \nu(\,\cdot\,| \S_{j})$ for any $k\ge0$.
\\
iv. If the chain has a qsd, then there is a qsd $\nu_\star$ such that the Yaglom limit of $\delta_1$ along $d$-periodic subsequences is given by
\[
\lim_k \delta_1\T_{d k +j-1}= \nu_\star(\,\cdot\,| \S_j).
\]
Moreover, for any other qsd $\nu$, one has $\nu_\star(\,\cdot\,| \S_j) \prec \nu(\,\cdot\,| \S_j)$ for all~$j$.
\\
v. If moreover $Q$ is such that $\nu\prec \nu'$ implies $a(\nu)\le a(\nu')$, then $\nu_\star=\nu_{\min}$.
\end{theorem}

Before giving the proof, we discuss the particular case of the $p$-$q$ random walk.
As an application of the above theorem, we prove convergence of the Yaglom limit to the minimal qsd based on monotonicity of trajectories.

The \emph{$p$-$q$ discrete-time random walk} is defined as follows.
Consider the periodic random walk with transition probabilities
\begin{equation}
\label{17}
\begin{array}[l]{l}
\Q(0,0)=1,\quad\Q(x,x-1)=q,\quad \Q(x,x+1)=p,\quad\hbox{for }x\ge 1,\\
\Q(x,y)=0, \ \text{ otherwise} ; \qquad p+q=1, \quad p<q.
\end{array}
\end{equation}
The chain has period $2$ and, starting from $\delta_1$, the walk visits odd sites at even times and vice-versa.
The qsd's for this random walk satisfy~(\ref{qsdpqr}) as before.
The minimal qsd $\nu_{\min}$ is given by~(\ref{a18}), and the remaining qsd's are given by~(\ref{qsd}).
The cyclic subclasses are $\S_1=2\N -1$ and $\S_2=2\N$.

\pagebreak[3]
\begin{corollary}
\label{c10}
Let $X_n$ be the discrete-time $p$-$q$ random walk with transition probabilities~\textnormal{(\ref{17})}.
The Yaglom limit of~$\delta_1$ converges along even and odd times to projections of $\nu_{\min}$ given by~\textnormal{(\ref{a18})}.
That is, for both~$j=1,2$,
\[
\lim_n \delta_1T_{2n+j-1} = \nu_{\min}(\,\cdot\,|\S_j)
.
\]
Moreover, $\nu_{\min}(\,\cdot\,|\S_j)\prec \nu(\,\cdot\,|\S_j)$ for any other qsd $\nu$.
\end{corollary}
\begin{proof}
By Theorem~\ref{t7}(iv) there is $\nu_{\star}$ with the above properties, and the Yaglom limit converges to projections of~$\nu_{\star}$ along even or odd subsequences.
By Theorem~\ref{t7}(v) we have $\nu_{\star}=\nu_{\min}$, concluding the proof.
\end{proof}

In order to prove Theorem~\ref{t7}, we start with some basic properties of qsd's for periodic chains.
For a probability~$\nu$ on~$S$, write $\nu = \sum_j m_j \nu_j$, where $\nu_j:=\nu(\,\cdot\,|S_j)$ and $\sum_j m_j =1$.
For shortness, let $S_j$, $m_j$ and $\nu_j$ be indexed by $j \in \Z_d$, so that $\S_{d+1}=\S_1$, etc.
Recall that $a(\nu)$ is defined in~(\ref{p35}).
\begin{lemma}
\label{p31}
Let $Q$ be the transition matrix for a $d$-periodic chain in $\S$ absorbed at 0. If $\nu$ is a qsd, then for each class $j$,
\( m_j a(\nu_j) = a(\nu)\, m_{j+1}, \)
and
\( \nu_jT_n = \nu_{j+n}, \)
for all $n\ge 0$.
In particular,
\( (a(\nu))^d = a(\nu_1)\dots a(\nu_{d}). \)
\end{lemma}
\begin{proof}
For any measure $\nu$ on $\S$, $\nu_j Q$ is supported on $S_{j+1} \cup \{0\}$, and thus
\( \nu_j Q_{|_S} = a(\nu_j)\, \nu_{j}T_1. \)
Hence,
\( \nu Q_{|_S} = \sum_{j} m_j\, \nu_j Q_{|_S} = \sum_{j} m_j\, a(\nu_j)\, \nu_{j}T_1 . \)
On the other hand, if $\nu$ is a qsd,
\( \nu Q_{|_S} = a(\nu)\,\nu = \sum_{j} a(\nu)\, m_j \nu_j , \)
and thus $\sum_{j} m_j\, a(\nu_j)\, \nu_{j}T_1 = \sum_{j} a(\nu)\, m_{j+1} \nu_{j+1} $.
Now notice that, for each class~$j$, the measures $\nu_jT_1$ and~$\nu_{j+1}$ are probabilities supported on~$S_{j+1}$, and these sets are disjoint.
Therefore, $\nu_j T_1 = \nu_{j+1}$, and $m_j a(\nu_j) = a(\nu)\, m_{j+1}$.
Iterating the former identity, we get $\nu_j T_n = \nu_j T_1^n = \nu_{j+n}$, and taking the product over~$j$ of both sides of the latter, we get $(a(\nu))^d = a(\nu_1)\dots a(\nu_{d})$.
\end{proof}

\begin{proof}
[Proof of Theorem~\ref{t7}]
Under the present assumptions on the transition matrix, Holley inequality holds for any pair of measures supported on the subspace of trajectories that start in a given cyclic subclass.
Therefore, parts~(i) and (ii) can be proved just as in the proof of Theorem~\ref{t1}.

To prove~(iii), let $\nu$ be a qsd.
By~(ii) and Lemma~\ref{p31},
\(
\delta_1 \, T_{n}
\prec
\nu_1 \, T_{n} =
\nu_{1+n}
.
\)
Claim (iii) follows by taking $n=dk+j-1$.

Proof of (iv).
As in the proof of Theorem~\ref{t1}(iv), by (i,iii) the limits
\[
\unu_j := \lim_k \delta_1\T_{d k +j-1}
\]
exist and satisfy
\[ \unu_j T_1 = \unu_{j+1},\]
and moreover $\unu_j \prec \nu_j$ for any qsd~$\nu$.
It remains to find the right constants~$m_j$ and show that $\nu_\star$ given by $\nu_\star = \sum_{j} m_j \unu_j$ is a qsd, that is, that there exists an $\alpha\in(0,1)$ such that
\(
\nu_\star Q_{|_S}= \alpha \nu_\star.
\)
Since $\unu_j Q_{|_S} = a(\unu_j) \unu_{j+1}$, the problem is equivalent to find $\alpha,m_1,\dots,m_{d}$ solving the system of $d$ linear equations given by $m_j a(\unu_j) = \alpha m_{j+1}$.
The system has a nonzero solution if and only if
\(
 \alpha^d = a(\unu_1) \cdots a(\unu_{d}).
\)
Choosing the positive $\alpha$ that satisfies this identity, the space of solutions is one-dimensional and its elements have coordinates which agree in sign.
Choosing~$m$ to be the unique solution to satisfy $\sum_{j} m_j = 1$, we have that $\nu_\star$ is a probability and moreover it is a qsd with $a(\nu_\star)=\alpha$, concluding the proof of~(iv).

Proof of (v). It suffices to prove that $a(\nu_{\star})\le a(\nu)$ for all qsd $\nu$.
By Lemma~\ref{p31},
\begin{align}
\nonumber
 ( a(\nu_{\star}))^d\, =\, a(\unu_1)\dots a(\unu_{d})\,\le \,a(\nu_1)\dots a(\nu_{d})\,=\,( a(\nu))^d,
\end{align}
where the inequality comes from (iv) and the hypothesis of~(v).
\end{proof}

\section*{Acknowledgements}

We thank Pablo Groisman for preliminary discussions and for calling our attention to~\cite{MR0368168}.
We thank an anonymous referee for several suggestions and comments which helped us to improve the presentation of this paper.
This work is partially supported by Consejo Nacional de Investigaciones Científicas y Técnicas, Agencia Nacional de Promoción Científica y Tecnológica and the project ``Mathematics, computation, language and the brain'', FAPESP project ``NeuroMat'' (grant 2011/51350-6).

\bibliographystyle{rollaalphasiam}
\bibliography{domination}

\end{document}